\documentclass[11pt]{amsart}

\usepackage{hyperref}
\usepackage{amsfonts,amssymb,amscd,amsmath,amsthm,enumerate,mathrsfs}
\usepackage{graphicx,type1cm,eso-pic}
\newcommand{\abs}{\vspace{12pt}}

\DeclareMathOperator{\id}{id}
\DeclareMathOperator{\Int}{Int}
\DeclareMathOperator{\Clos}{Clos}

\DeclareMathOperator{\Diff}{Diff}

\DeclareMathOperator{\card}{card}

\DeclareMathOperator{\Fix}{Fix}

\DeclareMathOperator{\la}{\langle}
\DeclareMathOperator{\ra}{\rangle}
\newcommand{\h}{h_{\text{top}}}
\newcommand{\skp}{{\langle .,. \rangle}}

\newcommand{\R}{\mathbb{R}}
\newcommand{\Z}{\mathbb{Z}}

\newcommand{\N}{\mathbb{N}}

\newcommand{\g}{\mathfrak{g}}

\renewcommand{\g}{\gamma}

\newcommand{\LL}{\mathcal{L}}

\renewcommand{\O}{\mathcal{O}}

\newcommand{\CC}{\mathcal{C}}
\newcommand{\RR}{\mathcal{R}}

\newcommand{\W}{\mathcal{W}}
\newcommand{\A}{{\mathcal{A}}}

\renewcommand{\AA}{{\mathbb{A}}}
\newcommand{\T}{{\mathbb{T}^2}}
\newcommand{\s}{{\mathbb{S}^2}}

\newcommand{\e}{\varepsilon}

\newcommand{\gam}{\gamma}
\newcommand{\Gam}{\Gamma}
\newcommand{\al}{\alpha}
\newcommand{\om}{\omega}
\newcommand{\lam}{\lambda}

\newcommand{\sig}{\sigma}

\theoremstyle{plain}
\newtheorem{defn}{Definition}[section]
\newtheorem{lemma}[defn]{Lemma}

\newtheorem{thm}[defn]{Theorem}
\newtheorem{cor}[defn]{Corollary}

\theoremstyle{definition}
\newtheorem{remark}[defn]{Remark}

\newcommand{\mane}{Ma\~n\'e}

\newcommand{\Poincare}{Poincar\'e }

\begin{document}

\hypersetup{pdftitle = {Ergodic components and topological entropy in geodesic flows of surfaces}, pdfauthor = {Jan Philipp Schr\"oder}}

\title[Ergodicity and entropy in geodesic flows of surfaces]{Ergodic components and topological entropy in geodesic flows of surfaces}
\author[J. P. Schr\"oder]{Jan Philipp Schr\"oder}
\address{Faculty of Mathematics \\ Ruhr University \\ 44780 Bochum \\ Germany}
\email{\url{jan.schroeder-a57@rub.de}}

\begin{abstract}
We consider the geodesic flow of reversible Finsler metrics on the 2-sphere and the 2-torus, whose geodesic flow has vanishing topological entropy. Following a construction of A. Katok, we discuss examples of Finsler metrics on both surfaces, which have large ergodic components for the geodesic flow in the unit tangent bundle. On the other hand, using results of J. Franks and M. Handel, we prove that ergodicity and dense orbits cannot occur in the full unit tangent bundle of the 2-sphere, if the Finsler metric has positive flag curvatures and at least two closed geodesics. In the case of the 2-torus, we show that ergodicity is restricted to strict subsets of tubes between flow-invariant tori in the unit tangent bundle of the 2-torus.
\end{abstract}

\maketitle


\section{Introduction and main results}

We recall the definitions of Finsler metrics and topological entropy.

\begin{defn} \label{def finsler}
Let $X$ be a manifold with zero section $0_X\subset TX$. A function $F:TX\to [0,\infty)$ is called a \emph{Finsler metric} on $X$, if
\begin{enumerate}
 \item $F$ is $C^\infty$ in $TX-0_X$,
 
 \item $F$ is positively homogeneous, i.e. $F(av)=a F(v)$ for all $v\in TX$ and all $a\geq 0$ and
 
 \item $F^2$ is strongly convex in the fibers, i.e. for all $x\in X, v\in T_xX-\{0\}$, the square $F^2|_{T_xX}$ has a positive definite Hessian at $v$.
\end{enumerate}
We denote by
\[ SX := \{v\in TX : F(v)=1\} \]
the \emph{unit tangent bundle} with respect to $F$. The \emph{geodesic flow} of $F$ (defined by the critical curves of the energy functional $\int_a^bF^2(\dot c) dt$) is denoted by
\[ \phi_F^t:SX \to SX. \]
$F$ is called {\em reversible}, if $F(-v)=F(v)$ for all $v$ and \emph{Riemannian}, if $F(v)=\sqrt{g(v,v)}$ for some Riemannian metric $g$ on $X$ and all $v\in TX$.
\end{defn}

\pagebreak

In order to characterize the dynamical complexity of flows or homeomorphisms, one can use topological entropy, which measures the exponential orbit growth.

\begin{defn}
Let $(X,d)$ be a compact metric space and $\phi^t:X\to X$ be a continuous or discrete dynamical system (i.e. a group action of homeomorphisms on $X$ by $\R, \Z$, respectively). For $T\geq 0 ,\e>0$ define
\[ s(T,\e) := \sup \card S , \]
where the supremum is taken over all subsets $S\subset X$ with the property, that
\[ \forall x,y\in S, x\neq y : \quad \max_{0\leq t \leq T} d(\phi^tx,\phi^ty)>\e  . \]
The \emph{topological entropy} of $\phi^t:X\to X$ is defined by
\[ \h(\phi^t) = \h(\phi^t,X) = \lim_{\e\to 0} ~ \limsup_{T\to\infty} ~ \frac{\log s(T,\e)}{T}. \]
\end{defn}

\begin{remark}\label{entropy poincare map}
In the following, we will use that the entropy of flows bounds the entropies of its first-return maps to \Poincare sections, provided the return-times are uniformly bounded (which will be the case in our applications), cf. the arguments in the proof of Proposition 2.1 in \cite{labrousse}. Hence, if the flow $\phi^t$ has zero entropy, the same is true for the first-return maps to \Poincare sections occurring below.
\end{remark}

In this paper, we consider Finsler metrics on orientable, closed surfaces $X$, whose geodesic flow has vanishing topological entropy. This last condition implies that the surface has to be the sphere $X=\s$ or the torus $X=\T$ (cf. Corollary 4.2 in \cite{dinaburg}).

\subsection{The 2-sphere}

In \cite{katok}, A. Katok gave the following two examples of Finsler metrics on the 2-sphere. In what follows, ergodicity of the geodesic flow $\phi_F^t$ is always meant with respect to the canonical {\em Liouville measure} in $S\s$, induced by the pullback of the canonical symplectic form in $T^*\s$ under the Legendre transform associated to $\frac{1}{2}F^2$.

\begin{thm}[Katok]\label{thm sphere katok}
\begin{enumerate}
 \item There exist non-reversible Finsler metrics on $\s$, arbitrarily close to the standard round metric, with only two closed geodesics and in particular vanishing topological entropy, whose geodesic flow is ergodic in the whole unit tangent bundle.
 
 \item There exist reversible Finsler metrics on $\s$, arbitrarily close to the standard round metric, whose geodesic flow has vanishing topological entropy and which admits two ergodic components $E_0,E_1$, such that $E_1=\{-v: v\in E_0\}$ and such that $E_0\cup E_1$ is arbitrarily large in the unit tangent bundle.
\end{enumerate}
\end{thm}

Note that the closeness of the Finsler metric to the round metric implies that the flag curvatures of $F$ are positive and in particular, any geodesic in the examples in Theorem \ref{thm sphere katok} possesses conjugate points. Recall that, if $c:\R\to\s$ is an $F$-geodesic, then for $t_0<t_1$ the points $c(t_0),c(t_1)$ are said to be {\em conjugate along $c$}, if writing $\pi:S\s\to\s$ for the bundle projection and $V_v = \ker d\pi(v) \subset T_vS\s$ for the vertical line bundle, we have for the geodesic flow
\[ d\phi_F^{t_1-t_0}(\dot c(t_0)) V_{\dot c(t_0)} = V_{\dot c(t_1)}. \]
We will prove the following theorem in Section \ref{section sphere}, which shows that example (2) of a reversible Finsler metric in Theorem \ref{thm sphere katok} is in some sense optimal. 

\begin{thm}\label{thm sphere}
 Let $F$ be a reversible Finsler metric on $\s$, possessing at least two closed geodesics (non-equal images in $\s$). Assume moreover, that every geodesic in $(\s,F)$ has a pair of conjugate points along itself. Then, if $\phi_F^t$ has a dense geodesic in $S\s$, we have
 \[ \h(\phi_F^t,S\s)>0. \]
\end{thm}

Hence, under the condition on conjugate points and as long as we have enough closed geodesics, dense geodesics -- and in particular ergodicity of the geodesic flow in $S\s$ -- imply the existence of hyperbolicity in the dynamical system $\phi_F^t:S\s\to S\s$, cf. Corollary 4.3 in \cite{katok1}.

Note that, in order to prove Theorem \ref{thm sphere}, we will use the Birkhoff annulus map. Hence our condition on the existence of at least two closed geodesics in $S\s$ implies by the celebrated result of J. Franks \cite{franks}, that we have in fact infinitely many closed geodesics, as in the unit tangent bundle the second closed geodesic becomes an interior periodic point in the Birkhoff annulus. Moreover, we remark that due to the results of J. Franks and M. Handel \cite{Fr-Ha}, which we apply in order to prove Theorem \ref{thm sphere}, there will under our assumptions be much more structure of $\phi_F^t:S\s\to S\s$, cf. Sections \ref{section Fr-Ha}, \ref{section sphere}.

It is well-known, that in Riemannian 2-spheres there always exist at least three (simple) closed geodesics \cite{grayson}, called the Lusternik-Schnirelmann geodesics. Hence, we obtain the following corollary of Theorem \ref{thm sphere}.

\begin{cor}\label{cor sphere}
 If $(\s,g)$ is a Riemannian 2-sphere with strictly positive curvature, then the existence of a dense geodesic in $S\s$ implies
 \[ \h(\phi_g^t,S\s)>0 \]
 and in particular the existence of a hyperbolic invariant set.
\end{cor}

Note that it is an open problem, whether there exist positively curved Riemannian 2-spheres with an ergodic geodesic flow. In this sense, Corollary \ref{cor sphere} shows that such examples would necessarily have chaotic behavior arising from hyperbolicity.

By the results of V. J. Donnay \cite{donnay}, there exist Riemannian metrics on $\s$ whose geodesic flow is ergodic in $S\s$ with a lot of negative curvature -- but these examples also have positive topological entropy. It is thus possible that Corollary \ref{cor sphere} also holds without the assumption of positive curvature, which, however, is a topic for future research.

\subsection{The 2-torus}

In this subsection we discuss the 2-torus $\T=\R^2/\Z^2$. The following theorem on geodesic flows on $\T$ with vanishing topological entropy is due to E. Glasmachers and G. Knieper \cite{glasm1}, \cite{glasm2} with an earlier version for monotone twist maps given by S. Angenent in \cite{angenent}. The letter $\pi$ stands in this paper for the canonical projections of tangent bundles, e.g. $\pi:T\T\to\T$. We write $c_v(t) = \pi \phi_F^tv$ for the unique geodesic with $\dot c_v(0)=v$. Note that then $\dot c_v(t)=\phi^t_Fv$.

\begin{thm}[Glasmachers, Knieper]\label{thm torus htop=0}
Let $F$ be a reversible Finsler metric on $\T$. If $\h(\phi_F^t)=0$, then for all $\rho\in S^1$ there exist $\phi_F^t$-invariant Lipschitz graphs $\Gam_\rho^-, \Gam_\rho^+ \subset S\T$ over $\T$ (meaning that $\pi|_{\Gam_\rho^\pm}:\Gam_\rho^\pm\to\T$ is a bi-Lipschitz homeomorphism) with the following properties.
\begin{enumerate}
 \item There exists a constant $D \geq 0$ depending only on $F$, such that for $v\in \Gam_\rho^\pm$, the lifted geodesics $\widetilde{c_v}(t)\in\R^2$ lie at distance at most $D$ from the straight euclidean line through $\widetilde{c_v}(0)$ with direction $\rho$ and escape to infinity along this line ($\widetilde{c_v}(t)$ moves in the direction $\pm\rho$, as $t\to \pm \infty$).
 
 \item If $\rho$ has irrational slope, then $\Gam_\rho^-=\Gam_\rho^+=:\Gam_\rho$.
 
 \item If $\rho$ has rational slope, then the intersection $\cap_{-,+}\Gam_\rho^\pm$ consists precisely of the velocities of the shortest closed geodesics in the prime homotopy class in $\R_{>0}\rho \cap\Z^2$. In particular, $\cap\Gam_\rho^\pm \neq \emptyset$. Moreover, each orbit in $\cup\Gam_\rho^\pm-\cap\Gam_\rho^\pm$ is heteroclinic between the two closest periodic geodesics in $\cap\Gam_\rho^\pm$. Here $\Gam_\rho^+$ is chosen in such a way that $\widetilde{c_v}(t)$ with $v\in\Gam_\rho^+$ is asymptotic in $\R^2$ to the right periodic geodesic as $t\to-\infty$ and to the left periodic geodesic as $t\to\infty$ (with respect to the orientation given by $\rho$); $\Gam_\rho^-$ has the opposite behavior.
 
 \item All orbits in $S\T - (\cup_{\rho\in S^1} \cup_{-,+} \Gam_\rho^\pm)$ are enclosed between the two graphs $\Gam_\rho^\pm$ with the same rational direction $\rho$ and when lifted to $\R^2$ tend to $\pm\infty$ along straight lines of direction $\pm\rho$, as $t\to\pm\infty$.
\end{enumerate}
\end{thm}

For intuition about the invariant sets in Theorem \ref{thm torus htop=0}, cf. Figure \ref{fig htop=0}.

The following terminology is motivated by the often found presence of elliptic closed geodesics in the complement of the graphs in Theorem \ref{thm torus htop=0}.

\begin{defn}\label{def elliptic tubes}
Let $F$ be a reversible Finsler metric on $\T$ with $\h(\phi_F^t)=0$. If $E$ is a connected component of $S\T - (\cup_{\rho\in S^1} \cup_{-,+} \Gam_\rho^\pm)$, enclosed by two graphs $\Gam_\rho^\pm$ of some rational direction $\rho\in S^1$, then we call $E$ an \emph{elliptic tube (of direction $\rho$)}.
\end{defn}

\begin{figure}[!htb]\centering
\includegraphics[scale=1.0]{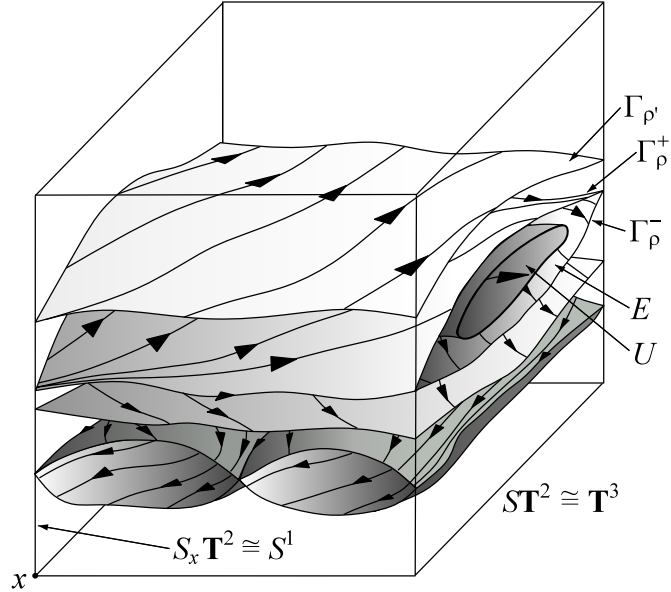}
\caption{The invariant graphs in the unit tangent bundle $S\T\cong \mathbb{T}^3$ occurring in Theorem \ref{thm torus htop=0}. The horizontal plane can be thought of as the base $\T$, $\pi$ being the vertical projection. Moreover, one can see an elliptic tube $E$ enclosed by two graphs $\Gam_\rho^\pm$ of direction $\rho=(1,0)$, containing a $\phi_F^t$-invariant subtube $U$, where $\phi_F^t|_U$ might be ergodic. \label{fig htop=0}}
\end{figure}

What remains open in Theorem \ref{thm torus htop=0} is any description of the dynamics of $\phi_F^t$ in the elliptic tubes $E\subset S\T$, besides the fact that they move to $\pm\infty$ in bounded distance from euclidean lines. Our first theorem shows that, even if $\h(\phi_F^t)=0$, elliptic tubes can contain complicated dynamical behavior of $\phi_F^t$. The example relies on the construction due to A. Katok \cite{katok}, which also led to the examples in Theorem \ref{thm sphere katok}. Recall that ergodicity is meant with respect to the Liouville measure in $S\T$.

\begin{thm}\label{thm torus katok}
There exist reversible Finsler metrics $F$ on $\T$ with vanishing topological entropy $\h(\phi_F^t)=0$, having an elliptic tube $E\subset S\T$ containing a $\phi_F^t$-invariant, open subtube $U\subset E$, such that $\phi^t_F|_U$ is ergodic. Moreover, the measure of $E-U$ can be made arbitrarily small.
\end{thm}

On the other hand, our next theorem shows that complicated behavior in all of $E$ is excluded. Philosophically speaking, the example in Theorem \ref{thm torus katok} is very degenerate in the center of the elliptic tube, while at its boundary the heteroclinics make up for a ``twist'', prohibiting ergodicity. We write $\O(\phi_F^t,v) = \{\phi^t_Fv:t\in\R\}\subset S\T$ for the orbit of $v$ and $\Clos A$ for the closure of a set $A$.

\begin{thm}\label{thm torus}
Let $F$ be a reversible Finsler metric on $\T$ with $\h(\phi_F^t)=0$ and suppose that $E\subset S\T$ is an elliptic tube. Then we have the following.
 \begin{enumerate}
  \item The set of vectors $v\in E$ with $\Clos\O(\phi_F^t,v)\cap\partial E \neq\emptyset$ has zero Liouville measure.
  
  \item There exists no orbit of $\phi_F^t$, which is dense in $E$. In particular, the restriction $\phi_F^t|_E$ is not ergodic with respect to the Liouville measure in $E\subset S\T$.
 \end{enumerate}
\end{thm}

Item (1) shows that almost every orbit in $E$ is bounded away from $\partial E$ by a positive constant. This indicates that there are large invariant sets in the interior of $E$, not touching the boundary $\partial E$. A possible picture would be a sequence of nested, invariant, closed tubes sitting inside the interior of the elliptic tube $E$. In general, we see by item (1), that for the Liouville measure $\mu_L$ in $S\T$ we have
\[ \lim_{\e\to 0} \mu_L(E-A_\e) = 0, \quad \text{where} \quad A_\e := \{ v\in S\T : d(\O(\phi_F^t,v), \partial E) \geq \e \} . \]
However, the sets $A_\e$ might a priori be quite exotic, opposed to the smoothly bounded invariant tubes, which we will find in the example in Theorem \ref{thm torus katok}.

\begin{remark}
Our results about geodesic flows in $\T$ are stated in terms of \emph{reversible} Finsler metrics. However, it might well be that Theorem \ref{thm torus} holds also in the non-reversible case. Apart from the fact that orbits outside the invariant graphs $\Gam_\rho^\pm$ tend to $\pm\infty$, Theorem \ref{thm torus htop=0} has been generalized to non-reversible Finsler metrics by the author in \cite{paper2}. The reversibility of $F$ is then only used to construct the \Poincare sections in Section \ref{section torus}. Note that general Finsler metrics can be used to describe the dynamics of Tonelli Lagrangian systems, cf. \cite{contreras}.
\end{remark}

\begin{remark}
The above results about geodesic flows in $\T$ remain true for {\em monotone twist maps} of the compact annulus. One can easily see that the example in Theorem \ref{thm torus katok} can be boiled down to a twist map. Moreover, for the proof of Theorem \ref{thm torus} we work with first-return maps to annuli, so we could have done the same with monotone twist maps.
\end{remark}

{\bf Structure of this paper.} In Section \ref{section katok}, we study rotational metrics on $\s$ and $\T$ and apply a result due to A. Katok from \cite{katok} to prove Theorem \ref{thm torus katok}. The main ingredient to prove Theorems \ref{thm sphere} and \ref{thm torus} are the results of J. Franks and M. Handel from \cite{Fr-Ha}, which we recall in Section \ref{section Fr-Ha}. In Sections \ref{section sphere} and \ref{section torus} we prove Theorems \ref{thm sphere} and \ref{thm torus}, respectively.

\section{Examples with large ergodic components}\label{section katok}

It will be convenient to work in the symplectic setting, as we are working with more than one Finsler metric. Let $X$ be a manifold with cotangent bundle $(T^*X,\om)$ endowed with its canonical symplectic form. Given a Finsler metric $F:TX\to [0,\infty)$, its {\em Legendre transform} $\LL : TX-0_X\to T^*X-0_X^*$ defined by $\LL(v) = \frac{1}{2} \frac{\partial F^2}{\partial v}(v)$ is a diffeomorphism. We can define the {\em dual Finsler metric} associated to $F$ by
\[ H := F\circ \LL^{-1}:T^*X\to [0,\infty) , \]
which has the same properties as a Finsler metric in Definition \ref{def finsler}. The Hamiltonian flow of $\frac{1}{2}H^2$ is conjugated via $\LL$ to the geodesic flow of $F$. Conversely, given a dual Finsler metric $H:T^*X\to [0,\infty)$, the analogously defined dual Legendre transform $\LL^*$ conjugates the Hamiltonian flow of $\frac{1}{2}H^2$ to the geodesic flow of the Finsler metric $H\circ (\LL^*)^{-1}$ and if $H$ was the dual Finsler metric associated to $F$ as above, then $H\circ (\LL^*)^{-1}=F$.

The origin of the examples in Theorems \ref{thm sphere katok} and \ref{thm torus katok} lies in the following special case of Theorem A from \cite{katok}.

\begin{thm}[Katok]\label{katok thm A}
 Let $X$ be a manifold with cotangent bundle $T^*X$ and $H_0,H_1:T^*X\to\R$ be commuting Hamiltonians (i.e. the Poisson bracket $\{H_0,H_1\}=0$), which are positively homogeneous of degree one. Let moreover $H_0$ be a dual Finsler metric. Assume that there exists an open subset $U\subset T^*X-0_X^*$, which is a fiberwise cone, i.e. $a\xi\in U$ if $\xi\in U, a > 0$, invariant under both Hamiltonian flows $\phi_{H_i}^t$ and suppose that $\phi_{H_i}^t|_U$ are periodic flows, i.e. $\phi_{H_0}^T|_U=\phi_{H_1}^S|_U=\id_U$ for some $T,S>0$. Write $X_H$ for the Hamiltonian vector field of a function $H$ and
 \[ U_0 := \{\xi\in U: X_{H_0}(\xi), X_{H_1}(\xi) \text{ linearly dependent} \} . \]
 Then for any $\e>0$, $k\in\N$ and any compact subset $K\subset T^*X$ there exists a function $H:U\to [ 0,\infty)$, such that the following conditions are satisfied:
 \begin{enumerate}
  \item $H$ is positively homogeneous and fiberwise strongly convex in the sense of Definition \ref{def finsler},
  
  \item $\|H-H_0\|_{C^k(K \cap U)} \leq \e$,
  
  \item in $\partial U \cup U_0$, the function $H$ coincides together with all its derivatives with a function of the form $H_0 + \al \cdot H_1$ with $|\al|\leq \e$.
  
  \item the Hamiltionian flow of $H$ is ergodic in each level set
  \[ H^{-1}(c)\cap (U-U_0), \qquad c>0 \]
  with respect to the volume defined by $\wedge^{\dim X}\om$,
  
  \item $\phi_H^t$ has no closed orbits in $U-U_0$.
 \end{enumerate}
\end{thm}

Note that, if $\psi^t:X\to X$ is a flow of isometries with respect to a Riemannian metric $g$ on $X$, then the lifted flow $d\psi^t:TX\to TX$ commutes with the geodesic flow $\phi_g^t$. Under the Legendre transform $\LL:TX\to T^*X$ associated to $g$ as above, the flow $\LL\circ d\psi^t \circ\LL^{-1}$ is the Hamiltonian flow of
\[ H_1:T^*X\to\R, \qquad H_1(\xi) := \xi(X_\psi(\pi \xi)) = g(X_\psi(\pi v) ,v ), \]
where
\[ \xi=\LL(v)=g(v,.), \qquad X_\psi=\textstyle \frac{d}{dt}\big|_{t=0}\psi^t. \]
Hence, if $H_0$ is the dual Finsler metric comming from $\sqrt g$, then $H_0^2/2,H_1$ are commuting Hamiltonians. Recall also that $H_0 \cdot X_{H_0} = X_{H_0^2/2}$ and hence, the Hamiltionian flows of $H_0$ and $H_0^2/2$ are reparametrisartions of each other, while due to homogeneity we have $\phi_{H_0}(a\xi)=a\phi_{H_0}^t(\xi)$ for $\xi\in T^*X,a > 0$. 

One immediately infers example (1) in Theorem \ref{thm sphere katok} by letting $H_0$ be the dual Finsler metric comming from the standard round metric on $\s$, $U=T^*\s-0_\s^*$ and $\psi^t$ being the periodic rotation of $\s\subset \R^3$ about the $x_3$-axis.

The goal of the rest of this section is to explain example (2) in Theorem \ref{thm sphere katok} and the example in Theorem \ref{thm torus katok}. Both can be studied in the setting of rotational metrics on the cylinder
\[ \CC : = \R/2\pi \Z\times \R . \]
Suppose $c=(c_1,c_2,c_3):I \to\R^3$ is a smooth curve, parametrized by euclidean arc-length and such that $c_2\equiv 0$ and $c_1>0$, defined on an interval $I\subset\R$. Then $c$ defines an immersed surface of revolution in $\R^3$ by
\[ \varphi:\R/2\pi \Z\times I \to \R^3, \qquad \varphi(x_1,x_2) = ( c_1(x_2) \cos x_1, c_1(x_2) \sin x_1, c_3(x_2)) . \]
The euclidean metric of $\R^3$ induces a Riemannian metric on $\R/2\pi \Z\times I $ by
\[ (\varphi^*\skp_{\R^3})_{x}(v,w)= \la v, G_0(x_2) w \ra_{\R^2}, \quad G_0(x_2) :=\begin{pmatrix} c_1^2(x_2) & 0 \\ 0 & |\dot c(x_2)|_{\R^3}^2 \end{pmatrix}. \]
We solve $h' = c_1\circ h$ for a function $h:J\to I$. By $|\dot c|_{\R^3}=1$ we obtain for $\tilde \varphi(x) := \varphi(x_1,h(x_2))$ that
\[ (\tilde \varphi^*\skp_{\R^3})_x = f^2(x_2) \cdot \skp_{\R^2}, \qquad f :=c_1\circ h > 0 \]
on $\R/2\pi \Z \times J$ (note that $h'>0$ due to $c_1>0$, so $\tilde \varphi$ is again an immersion). Such a metric $f^2(x_2) \cdot \skp_{\R^2}$ is called a {\em rotational metric}. In the following we write $\skp, |.|$ for the euclidean metric and its norm on $\R^2$, also defined in standard coordinates on $(\R^2)^*$. Note that, if $f^2\skp$ is a rotational metric, its dual Finsler metric is given by
\[ H_0:T^*\CC\to\R, \qquad H_0(\xi) := \frac{1}{f((\pi \xi)_2)} \cdot |\xi|. \]
A $2\pi$-periodic flow of isometries with respect to a rotational metric on $\CC$ is given by $\psi^tx=(x_1+t,x_2)$. Hence, as described above, the Hamiltonian
\[ H_1:T^*\CC\to\R, \qquad H_1(\xi) := \xi(X_\psi(\pi \xi)) = \xi_1 \]
commutes with $H_0$. Here we wrote $\xi=\xi_1 dx_1 + \xi_2 dx_2\in (\R^2)^*\cong T_{\pi\xi}^*\CC$.

For the case of the 2-sphere, we take the half circle $c(t)=(\cos t,0,\sin t)$ with $t\in (-\pi/2,\pi/2)$. One can easily check that one obtains
\[ h:\R\to (-\pi/2,\pi/2) , \qquad h(t) = 2\arctan (e^t)-\pi/2 \]
and as a function $f=\cos \circ h$
\[ f_0(t) := \frac{2e^t}{1+e^{2t}}, \qquad f_0:\R\to (0,\infty). \]
Hence, the $2\pi$-periodic geodesic flow of the round sphere minus the north and south pole can be described by the rotational metric $f_0^2\cdot \skp$ on the cylinder $\CC$. Note that $f_0(-t)=f_0(t)$ and that in $[0,\infty)$, the function $f_0$ is strictly decreasing.

In order to obtain examples for the 2-torus, we can choose any periodic function $f:\R\to(0,\infty)$. Then the so obtained rotational metric descends to a metric on a torus. In what follows, we will take for the torus-case a function $f$ of some period $L>0$ and assume that for a small $\e>0$ the function $f$ coincides on $[-L/2+\e,L/2-\e]$ with the function $f_0$ obtained from the round 2-sphere.

\begin{lemma}\label{periodic flow}
Consider a rotational metric $f^2 \skp$ on the cylinder $\CC=\R/2\pi \Z\times \R$ and assume that
\[ \exists ~ b>0 : \qquad  f|_{[-b,b]} = f_0|_{[-b,b]} . \]
For $a\in (0,b)$, letting $H_0=\frac{1}{f}|.|, ~ H_1=\xi_1$ as above, set
\[ U_a := \left \{ ~ \xi \in T^*\CC  ~ : ~ |(\pi\xi)_2|\leq a , ~ \textstyle\frac{H_1(\xi)}{H_0(\xi)} \geq f_0(a) ~ \right \} . \]
Then the sets $U_a$ are invariant under both $\phi_{H_i}^t, i=0,1$ and $\phi_{H_i}^t|_{U_a}$ are $2\pi$-periodic flows. Choose $0<a_0<a_1<b$ and two functions $\chi:\CC\to \R,~ \eta:\R\to\R$, where $\eta$ is smooth, with
\[ \chi(x_1,x_2)= \begin{cases} 0 & : |x_2|>b \\ 1 & : |x_2| \leq b \end{cases} , \qquad  \eta(t) = \begin{cases} 0 & : t \leq f_0(a_1) \\ 1 & : t \geq f_0(a_0) \end{cases} \]
and set
\[ \psi: T^*\CC \to \R, \quad \psi(\xi) := \chi(\pi \xi) \cdot \eta\left(\textstyle \frac{H_1(\xi)}{H_0(\xi)}\right) \cdot H_1(\xi). \]
Then $\psi$ is smooth in $T^*\CC-0_\CC^*$. For $|\al|$ small consider the dual Finsler metric
\[ H_\al (\xi) := H_0(\xi) + \al \cdot \psi(\xi) \]
on $\CC$. Then the Hamiltonian flow of $H_\al$ is completely integrable and
\begin{align}\label{F cut-off formula}
 & H_\al(\xi)= H_0(\xi) && \forall \xi\in T^*\CC-U_{a_1}, \\ \nonumber
 & H_\al(\xi)= H_0(\xi) + \al H_1(\xi) && \forall \xi\in U_{a_0} .
\end{align}
\end{lemma}

\begin{proof}
The invariance of $U_a$ under $\phi_{H_1}^t$ is trivial. For invariance under $\phi_{H_0}^t$ we only have to check that $|(\pi \phi_{H_0}^t\xi)_2|\leq a$ for all $t\in \R$, if $\xi\in U_a$. But for this, observe that for $\xi\in U_a$ by $f_0(t)=f_0(|t|)$
\begin{align*}
 f_0(a) \leq \textstyle \frac{H_1(\phi_{H_0}^t\xi)}{H_0(\phi_{H_0}^t\xi)} = f_0((\pi \phi_{H_0}^t\xi)_2)\frac{(\phi_{H_0}^t\xi)_1}{\sqrt{(\phi_{H_0}^t\xi)_1^2+(\phi_{H_0}^t\xi)_2^2}} \leq f_0(|(\pi \phi_{H_0}^t\xi)_2|).
\end{align*}
As $f_0$ is strictly decreasing in $[0,\infty)$ we have $|(\pi \phi_{H_0}^t\xi)_2|\leq a$ for all $t$. The periodicity of $\phi_{H_1}^t$ is again trivial and for $\phi_{H_0}^t|_{U_a}$ this follows, since the geodesic flow of the round 2-sphere is $2\pi$-periodic in its unit tangent bundle.

The equations \eqref{F cut-off formula} are obvious from the definition of $U_a$. To see that $\psi$ is smooth, just observe that $\eta\left(\textstyle \frac{H_1(\xi)}{H_0(\xi)}\right)=0$ for $|(\pi \xi)_2|\in (a_1,b)$:
\begin{align*}
 \textstyle \frac{H_1(\xi)}{H_0(\xi)} = f_0((\pi \xi)_2)\frac{\xi_1}{\sqrt{\xi_1^2+\xi_2^2}} \leq f_0((\pi \xi)_2) \leq f_0(a_1) 
\end{align*}
due to the monotonicity of $f_0$. For the integrability of $\phi_{H_\al}^t$ just observe that $H_\al$ is defined in terms of the commuting integrals $H_0,H_1$, which are well-known to be independent almost everywhere.
\end{proof}

\begin{remark}
 The set $U_a$ is a cone around $dx_1$ in each $T_x^*\CC$. For $|x_2|> a$ the cone is empty, while for $|x_2|=a$ it is a line and for $x_2=0$ the cone is opened the widest. If $f=f_0$, then for $a\to \infty$ the complement
 \[ T^*\s-\{ -\xi,\xi : \xi\in U_a \} \]
 becomes arbitrarily small by $f_0(a) \to 0$ for $|a|\to\infty$. If $f$ is periodic with period $L$ and if $a=L/2-\e$, then in the quotient $T_L^2=\R^2/(2\pi \Z \oplus L \Z)$, the set $U_a$ becomes arbitrarily large for $\e\to 0$ in the dual elliptic tube $E^* \subset S^*T_L^2$ of direction $e_1$ given by the connected component of $U_a$ in
 \[ \{ \xi \in T^*T_L^2 : H_0(\xi) =1 , H_1(\xi) > \min f \}. \]
\end{remark}

We can now readily apply Theorem \ref{katok thm A} to prove the existence of the reversible Finsler metrics in Theorems \ref{thm sphere katok} (2) and \ref{thm torus katok}.

\begin{proof}[Proof of Theorems \ref{thm sphere katok} (2) and \ref{thm torus katok}]
\underline{Step 1} (existence of a non-reversible \linebreak dual Finsler metric $H$). We work in the setting described above, i.e. we are given a rotational metric $f^2\skp$ on $\CC$, which coincides with $f_0^2\skp$ in $\R/2\pi \Z\times [-L/2+\e,L/2-\e]$ for some $L>0$ and a small $\e>0$. By Theorem \ref{katok thm A} applied to the set $U_{a_0}$ in Lemma \ref{periodic flow} with $a_0 \in (L/2-2\e,L/2-\e)$, we find a dual Finsler metric $H :U_{a_0}\to[0,\infty)$ with only one periodic orbit in $U_{a_0}$ (the ``equator'' $\R/2\pi \Z\times \{0\}$) and having an ergodic Hamiltonian flow in  each level set $U_{a_0}\cap H^{-1}(c)$. In particular, the topological entropy of the Hamiltonian flow of $H$ in $U_{a_0}\cap H^{-1}(c)$ vanishes, as there is only subexponential growth of closed orbits, cf. Corollary 4.4 in \cite{katok1}. Moreover, $H$ coincides with a dual Finsler metric of the form $H_0+\al H_1$ in $\partial U_{a_0}$ together with all its derivatives, which by Lemma \ref{periodic flow} can be extended to a dual Finsler metric defined in all of $T^*\CC$. This dual Finsler metric, denoted again by $H$ has a Hamiltonian flow with vanishing topological entropy: In $U_{a_0}$ this was observed before and in $T^*\CC-U_{a_0}$ this follows from the integrability of the Hamiltonian flow of $H_\al$ in Lemma \ref{periodic flow}, cf. Theorem 1 in \cite{paternain}.

\underline{Step 2} (make $H$ reversible). By Lemma \ref{periodic flow} we have $H = H_0$ in the neighborhood $T^*\CC-U_{a_1}$ of $\R dx_2$ in each $T_x^*\CC$. Hence we can define a new dual Finsler metric
\[ H'(\xi) := \begin{cases} H(\xi) & : H_1(\xi) \geq 0 \\ H(-\xi) & : H_1(\xi) < 0 \end{cases} \]
on $\CC$, which is now a reversible dual Finsler metric. The metric is unchanged in the $\phi_H^t$-invariant set $\{ H_1(\xi) \geq 0\}$ and in $\{ H_1(\xi) \leq 0\}$ the Hamiltionian flow of $H'$ is just the reversed flow of $H$ from $\{ H_1(\xi) \geq 0\}$. Hence, this dual Finsler metric has all the desired properties. Translating into the Lagrangian setting in $T\CC$ as described at the beginning of this section, we obtain a Finsler metric $F$ on $\CC$.
\end{proof}

\section{The results of J. Franks and M. Handel}\label{section Fr-Ha}

In this section we recall (and slightly adjust) results of J. Franks and M. Handel from \cite{Fr-Ha}. Let us review the setting of \cite{Fr-Ha}.

Let $\mu$ be a measure on the 2-sphere $\s$ topologically conjugate to the Lebesgue measure (i.e. there exists a homeomorphism of $\s$ conjugating $\mu$ to the Lebesgue measure). Let $N$ be a surface diffeomorphic to $\s$ with $n$ disjoint, smoothly bounded, open discs removed. Collapsing each boundary circle $\partial_iN$ of $N$ into a point $p_i\in \s$ defines a $C^0$ quotient map $\pi_N:N\to \s$, whose restriction $\Int N\to \s-P$ is a $C^\infty$ diffeomorphism, where $P=\{p_1,...,p_n\}$. If $\phi:N\to N$ is an orientation-preserving $C^\infty$ diffeomorphism, leaving each boundary component $\partial_iN$ invariant, we can define a homeomorphism $\psi:\s\to \s$ by $\psi\circ\pi_N=\pi_N \circ \phi$, such that $P\subset \Fix(\psi)$. We denote by $\Diff(\s,P,\mu)$ the set of all so obtained homeomorphisms $\psi:\s\to \s$, that in addition preserve the measure $\mu$.

For $\psi\in\Diff(\s,P,\mu)$ set
\begin{align*}
 M := \s-\Fix(\psi), \qquad f := \psi|_M
\end{align*}

\begin{defn}[free disc recurrence, cf. Definition 1.1 in \cite{Fr-Ha}]\label{def disc rec}
 A (topological) open disc $B\subset M$ is a \emph{free disc} for $f$, if $f(B)\cap B=\emptyset$. A point $x\in M$ is called \emph{free disc recurrent} for $f$, written $x\in W$, if there exists $n\in\Z-\{0\}$ and a free disc $B$ for $f$ with $x, f^n(x)\in B$. A point is called \emph{weakly free disc recurrent}, written $x\in\W$, if $x\in \Int_M(\Clos_M(W_0))$ for some connected component $W_0$ of $W$.
\end{defn}

Note that $\W$ contains the full-measure set of birecurrent points for $f$ in $M$, and that $\W$ is open and dense in $M$.

We write
\[ \AA := \R/\Z\times[0,1], \qquad \tilde\AA := \R\times[0,1], \qquad \Int\AA=\R/\Z\times(0,1). \]

\begin{lemma}[annular compactification, cf. Notation 2.7 in \cite{Fr-Ha}]\label{lemma prime ends}
 If $U\subset M$ is an $f$-invariant, open annulus, then there exists a homeomorphism $h_U: \AA \to\AA$ (called the \emph{annular compactification}) of the closed annulus, which is smoothly conjugated to $f|_U$ in $\Int\AA$. If in $\s$, one end $\partial_1U$, say, of $\Clos_\s U$ contains more than one point and in addition one point, which is fixed by $\psi$, then also $\partial_1\AA$ contains a fixed point of $h_U$.
\end{lemma}

The last assertion follows from the properties of the prime-end compactification (cf. \cite{mather-prime-end}).

\begin{defn}[rotation number, cf. Definition 2.1 in \cite{Fr-Ha}]
Let $h:\AA\to\AA$ be a homeomorphism of the closed annulus and $\tilde h:\tilde\AA\to \tilde\AA$ a lift to the universal cover. We write $p_1:\tilde\AA\to\R$ for the projection to the $\R$-factor and setting
\[ \tilde \tau_{\tilde h}(\tilde x) := \lim_{n\to\infty}\frac{p_1(\tilde h^n(\tilde x))-p_1(\tilde x)}{n} , \quad \tilde x \in \tilde\AA \]
(if the limit exists), we define the \emph{rotation number}
\[ \rho_h(x) \in \R/\Z, \quad x\in\AA \]
to be the projection of $\tilde \tau_{\tilde h}(\tilde x)$ to $\R/\Z$, where $\tilde x$ projects to $x$.
\end{defn}

Clearly, $\rho_h$ is independent of the choice of the lift $\tilde h$, invariant under $h$ and is defined almost everywhere in $\AA$ by Lemma 2.2 in \cite{Fr-Ha}.

We can now state the two theorems from \cite{Fr-Ha}, which we are going to use, describing the structure of $\psi$-invariant sets in $\s$.

\begin{thm}[Franks, Handel]\label{thm 1.2 Fr/Ha}
 Let $\psi\in\Diff(\s,P,\mu)$ have infinite order and $\h(\psi)=0$. Then there exists a countable family $\A$ of pairwise disjoint, $f$-invariant, open annuli $U\subset M$ with the following properties:
 \begin{enumerate}
  \item the union $\bigcup_{U\in\A} U$ equals the set $\W$ of weakly free disc recurrent points for $f$ in $M$,
  
  \item $\A$ is the set of maximal $f$-invariant, open annuli in $M$, i.e. if $V\subset M$ is an $f$-invariant, open annulus, then there exists $U\in \A$ with $V\subset U$.
 \end{enumerate}
\end{thm}

\begin{thm}[Franks, Handel]\label{thm 1.4 Fr/Ha}
 Let $\psi\in\Diff(\s,P,\mu)$ have infinite order and $\h(\psi)=0$, let $\A$ be given by Theorem \ref{thm 1.2 Fr/Ha} and for $U\in\A$ let $h_U:\AA\to \AA$ be the annular compactification of $f|_U : U \to U$. Then
 \begin{enumerate}
  \item the rotation number $\rho_{h_U}:\AA\to \R/\Z$ is well-defined and continuous everywhere, 
  
  \item if $\Fix(\psi)\subset \s$ contains at least three points, then $\rho_{h_U}$ is non-constant.
  
 \end{enumerate}
\end{thm}


In our applications, $\psi\in\Diff(\s,P,\mu)$ is obtained from a first-return map $\phi:N\to N$ of a \Poincare section $N$, and in this situation we will have an invariant measure $\nu$ defined by a smooth volume form only in $\Int N \cong \s-P$. Hence, the following observation will be useful.

\begin{lemma}\label{lemma mu}
 If $\nu$ is a measure in $\s-P$ induced by a smooth volume form defined in $\s-P$, such that $\nu(\s-P)<\infty$, then the measure $\mu$ in $\s$, defined by
 \[ \mu(A)=\nu(A-P) , \]
 is topologically conjugate to a Lebesgue measure (i.e. the properly rescaled standard Lebesgue measure).
\end{lemma}

\begin{proof}
Let $D$ be the $n$-dimensional closed unit ball and $\mu_L$ be the (outer) Lebesgue measure on $D$. Recall the following theorem due to J. C. Oxtoby and S. M. Ulam, cf. Theorem 2 in \cite{oxtoby_ulam}. A finite outer measure $\mu$ on $D$
is topologically conjugate to $\frac{\mu(D)}{\mu_L(D)} \cdot \mu_L$ if and only if $\mu$ satisfies the following conditions:
\begin{enumerate}
 \item (Caratheodory's condition) If $A,B\subset D$ with $\inf\{d(x,y):x\in A,y\in B\}>0$, then $\mu(A\cup B)=\mu(A)+\mu(B)$,
 
 \item (regularity) $\mu(A) = \inf_{U\supset A \text{ open}} \mu(U)$,
 
 \item (positive on open sets) $\mu(U)>0$ for $U\neq \emptyset$ and $U$ open,
 
 \item (no atoms) $\mu(\{\text{pt}\}) = \mu(\partial D)=0$.
\end{enumerate}
Moreover, the homeomorphism $h:D\to D$ between $\mu,\mu_L$ can be chosen to satisfy $h|_{\partial D}=\id_{\partial D}$.

We now return to our measure $\mu$ on $\s$, let $\mu_L$ be the Lebesgue measure on $\s$ and set $\lam = \frac{\mu(\s)}{\mu_L(\s)}$. Consider the equator $\gamma = \s\cap\{x_3=0\}$ and rotate $\gamma$ about the $x_1$-axis. Then for each $t\in [0,2\pi]$, we can write $\s$ as the union of to compact discs $D_t^0, D_t^1$, that intersect in the rotated equator $\gamma_t$. Assuming that e.g. $\mu(D_t^0) \leq \mu(D_t^0)$ in $t=0$, we find the opposite inequality after time $t=\pi$. By continuity of $t\mapsto \mu(D_t^i)$ we find some $t_0\in[0,\pi]$, where both discs $D_{t_0}^0,D_{t_0}^1$ have the same $\mu$-area. Obviously, by our assumptions, $\mu$ restricted to the discs $D_{t_0}^i$ satisfies items (1)-(4) above, and hence we can apply the theorem of Oxtoby and Ulam to $\mu|_{D_{t_0}^i}$ to obtain homeomorphisms $h_i:D_{t_0}^i\to D_{t_0}^i$ conjugating $\mu$ to $\lam\cdot \mu_L$ in $D_{t_0}^i$. By the additional assertion that $h_i|_{\partial D_{t_0}^i}$ is the identity, we obtain the desired homeomorphism of $\s$.
\end{proof}

In order to apply Theorems \ref{thm 1.2 Fr/Ha} and \ref{thm 1.4 Fr/Ha} to \Poincare sections in $S\T$ for Finsler metrics on the 2-torus, we need the following observation, which under certain conditions allows the boundary circles of the surface $N$ above to be only continuous, instead of $C^\infty$.

\begin{lemma}\label{FrHa works on T^2}
 Let $\gamma_1,...,\gamma_n\subset \s$ be disjoint, continuous, simple, closed curves and let $N\subset \s$ be the compact surface obtained from $\s$ by cutting out interiors of the $\gamma_i$. Then Theorems \ref{thm 1.2 Fr/Ha}, \ref{thm 1.4 Fr/Ha} continue to hold for $\mu$-preserving homeomorphisms $\psi:\s\to\s$ obtained from diffeomorphisms $\phi:N\to N$ by collapsing each $\gamma_i$ into a fixed point $p_i\in P$ as above with the additional condition that $\phi:N\to\s$ extends to a $C^\infty$ embedding of an open neighborhood $U\subset\s$ of $N$ into $\s$.
\end{lemma}

\begin{proof}
 The only two places, where Franks and Handel use the smoothness of $\partial N$ in \cite{Fr-Ha} is to prove the following two statements:
 \begin{enumerate}
  \item Let $\sig:[0,1]\to N$ be a smooth curve segment and $\ell$ denote the length with respect to any Riemannian metric in $U$. Then
 \[ \limsup_{n\to\infty} \frac{1}{n} \log \ell(\phi^n(\sig)) \leq \h(\phi). \]
 
  \item There exists a finite family $\RR$ of essential, non-peripheral, non-parallel, simple, closed curves in $\s-\Fix(\psi)$, such that the homeomorphism $\psi\in \Diff(\s,P,\mu)$ is isotopic relative to $\Fix(\psi)$ to a composition of non-trivial Dehn twists in the elements of $\RR$.
 \end{enumerate}
We first discuss item (2). Considering the embedding $\phi:U\to \s$, we show how to find a $C^\infty$ diffeomorphism $F:\s\to \s$, such that $F|_N=\phi|_N$; then we can proceed as in Section 4 of \cite{Fr-Ha}. For the existence of $F$, choose within $U$ $n$ smoothly bounded, compact annuli $A_1,..., A_n$, each $A_i$ containing $\gamma_i$ in its interior and bounding open discs $D_i$ in $\s-N$. It is well-known that each $\phi|_{A_i}$ is isotopic to the inclusion $A_i\hookrightarrow \s$ and by the isotopy extension theorem (Theorem 1.4 on p. 180 of \cite{hirsch}), this isotopy can be extended to a smooth isotopy of the union $D_i\cup A_i$, ending in $\phi$ in each $A_i$. Taking the time-1-map of these isotopies in the $D_i$, we have found $F$. Item (1) now follows from applying Theorem 1.4 in \cite{yomdin} to $F$ and smooth curves $\sig$ in the $F$-invariant set $N\subset\s$.
\end{proof}

\section{Non-ergodicity in the case of the 2-sphere}\label{section sphere}

In this section we let $(\s,F)$ be the 2-sphere with a reversible Finsler metric $F$. Moreover, we will in the following assume that every geodesic of $F$ has conjugate points and claim that a dense geodesic in $S\s$ implies $\h(\phi_F^t)>0$ (Theorem \ref{thm sphere}). In order to prove this, we want to apply the results of J. Franks and M. Handel from Section \ref{section Fr-Ha}, i.e. we need \Poincare sections in $S\s$. A very classical construction is due to G. D. Birkhoff (cf. Section VI.10 of \cite{birkhoff}), which we will recall now.

It is well-known, that there always exists a simple, closed geodesic $c:\R/T\Z\to \s$, $T>0$ being the minimal period of $c$, by minimax methods, cf. Section 15-19 of \cite{birkhoff1}. Letting $N:\R/T\Z\to S\s$ be a unit vector field along $c$, orthogonal to $\dot c$ with respect to the standard round metric denoted by $\skp$, we define a smoothly bounded, compact annulus
\[ A := \{ v\in S\s ~|~ \exists t\in \R/T\Z : \pi v = c(t) , ~ \la v, N(t) \ra >0 \} \subset S\s, \]
which we call the {\em Birkhoff annulus} with base geodesic $c$ (in direction $N$).

\begin{lemma}[Birkhoff, Bangert]\label{lemma birkhoff}
 If $F$ is a reversible Finsler metric on $\s$, then the Birkhoff annulus $A\subset S\s$ with base geodesic $c:\R/T\Z\to \s$ is everywhere transverse to the generator of the geodesic flow $\phi_F^t:S\s\to S\s$ in the interior of $A$. Moreover, if every geodesic of $F$ possesses conjugate points, then every geodesic in $S\s-\{ \dot c(t), -\dot c(t) : t\in \R/T\Z \}$ hits $A$ in uniformly bounded positive and negative times.
\end{lemma}

We give a proof along the lines of V. Bangert's arguments, cf. Section 4 of \cite{bangert} (which discusses the Riemannian case).

\begin{proof}
 For the transversality observe that for $v\in \Int A$ we have $\frac{d}{dt}\big|_{t=0} \pi\circ \phi_F^tv=v$, which is transverse to $\pi(A)=c(\R/T\Z)$ by definition.
 
 Let $L$ be the supremum of times that unit speed geode\-sics take to hit $A$ and assume $L=\infty$. Then there exists a sequence of arc-length geodesic segments $c_n:[0,L_n]\to \s$ with $L_n\to\infty$, disjoint from the base geodesic $c:\R/T\Z\to \s$ of $A$. Letting $v_n:= \dot c_n(L_n/2)\in S\s$, take a convergent subsequence $v_n\to v$, then $c_v:\R\to\s$ is entirely disjoint from $c$. By the assumption that $c$ possesses conjugate points, it is not possible for $c_v$ to come arbitrarily close to $c$ without intersecting it: for then the geodesic flow would take the orbit $\dot c_v(t)$ across $c$ due to the existence of conjugate points, which means that $\phi_F^t$ rotates along the closed orbit $\dot c(\R/T\Z)$. Hence $\inf_{t\in\R}d(c(\R/T\Z),c_v(t))>0$ and the pair of geodesics $c,c_v$ bounds an open annulus $U\subset \s$, which is locally geodesically convex, as it is bounded by (parts of) geodesics.
 
 Fixing $k\geq 1$, take a sequence $\g_n^k$ of smooth, simple, closed curves in the prime homotopy class of the $k$-fold cover $U^k$ of $U$, such that the $F$-lengths $l_F(\g_n^k)$ decrease with $n\to\infty$ to the infimum of lengths of such curves. For each $n$ deform $\g_n^k$ into a closed geodesic in $U^k$ by means of the curve shortening flow for reversible Finsler metrics \cite{angenent1} (the curves stay in $U^k$ due to local geodesic convexity). In the limit, we obtain a shortest, simple, closed geodesic $\g^k$ in the prime homotopy class of $U^k$.
 
 If the boundary component of $U^k$ corresponding to $c_v$ is not smooth, then the smooth curve $\g^k$ is disjoint from this boundary component. On the other hand, $\g^k$ cannot be equal to the $k$-th iterate $c^k$ of $c$ for large $k$ (and by the same reasoning not equal to the other boundary component of $U^k$, if it is smooth). For this, observe that due to the existence of conjugate points along $c$, we can find on both sides of $c^k$ smooth, closed curves close to $c^k$, which are shorter than $c^k$ (cf. Lemma 2 in \cite{bangert} or for the Finsler case the techniques in Chapter 7.4 of \cite{bao-chern-shen}). It now follows from classical arguments of G. A. Hedlund (cf. Section 5 in \cite{hedlund}), that $\g^k$ is in fact prime periodic in $U$ and locally minimizing on arbitrarily long subsegments and hence it has to be free of conjugate points. This contradicts our hypothesis on the existence of conjugate points along every geodesic.
\end{proof}

As a corollary, we obtain a smooth first-return map
\[ \phi: \Int A \to \Int A, \]
which (as a map coming from a Hamiltonian flow) is well-known to preserve a smooth area form, also defined in the interior $\Int A$. This is sometimes called the {\em Birkhoff annulus map}. In order to apply the results from Section \ref{section Fr-Ha}, we need a smooth continuation of $\phi$ to all of $A$.

\begin{lemma}\label{birkhoff smooth}
 If $F$ is a reversible Finsler metric on $\s$ with conjugate points along every geodesic, then the Birkhoff annulus map extends to a $C^\infty$ diffeomorphism $\phi : A\to A$.
\end{lemma}

\begin{proof}
 By Lemma \ref{lemma birkhoff}, the first-return time $\tau:\Int A\to \R$ of a point $x\in \Int A$ to $A$ under $\phi_F^t$ is uniformly bounded from above. Let us investigate $\tau$ near the boundary $\dot c(\R/T\Z)$ of $A$, the other part $-\dot c(\R/T\Z)$ being treated analogously. If $\tau$ extends to a $C^\infty$ function $A\to\R$ then $\phi(v)=\phi_F^{\tau(v)}v$ is smooth in all of $A$.
 
 We identify a strip in $\s$ around $c(\R/T\Z)$ with $U := \R/T\Z \times (-\e,\e) \ni (t,s)$ in such a way that $c(t)\cong (t,0)$. For $r\in \R$ let $V(t,s,r)$ be the $F$-unit vector in $T_{(t,s)}U$, which makes an angle $r \mod 2\pi$ with the vector $\partial_t\in T_{(t,s)}U$ (angles with respect to the euclidean metric in $U$). Then $V:M \to SU$ is a diffeomorphism between $M:=U\times \R/2\pi\Z$ and the neighborhood $SU\subset S\s$ of $\dot c(\R/T\Z)$. Let $X:S\s\to TS\s$ be the generator of the geodesic flow $\phi_F^tv=\dot c_v(t)$, then $V^*X$ being the pullback of $X$ to a vector field on $M$, we can write
 \begin{align*}
  V^*X &=((V^*X)_t,(V^*X)_s,(V^*X)_r)\in \R^3\cong T_{\pi V^*X}M, \\
  V &=(V_t,V_s)\in \R^2\cong T_{\pi V}U .
 \end{align*}
 Using linearity of the $s$-projection and $\pi V(t,s,r)=(t,s)$, we obtain
 \begin{align*}
  (V^*X(t,s,r))_s & = \left(dV^{-1}(V(t,s,r)) \textstyle \frac{d}{d\tau}\big|_{\tau=0} \dot c_{V(t,s,r)} (\tau) \right)_s \\
  & = \textstyle \frac{d}{d\tau}\big|_{\tau=0} \left(V^{-1} \circ \dot c_{V(t,s,r)} (\tau) \right)_s \\
  & = \textstyle \frac{d}{d\tau}\big|_{\tau=0} \left(c_{V(t,s,r)} (\tau) \right)_s  = \left(V(t,s,r) \right)_s .
 \end{align*}
 Hence
 \begin{align*}
  \textstyle\frac{d}{dr}\big|_{r=0} (V^*X)_s(t,0,r) & = \textstyle\frac{d}{dr}\big|_{r=0} V_s(t,0,r) \neq 0
 \end{align*}
 by definition of $V$. Observe moreover, that
 \[ V^{-1}(A) = \R/T\Z \times \{0\} \times [0,\pi].  \]
 The lemma now follows from Lemma \ref{lemma umberto} below.
\end{proof}

\begin{lemma}\label{lemma umberto}
Let $\psi^\tau : \R^3\to \R^3$ be a local $C^\infty$ flow, such that writing $(t,s,r) \in \R^3$ for the coordinates, we have
 \begin{enumerate}
  \item $\psi^\tau(t,0,0)=(t+\tau,0,0)$, i.e. $\gam:\R \to M$ with $\gam(\tau)=(\tau,0,0)$ is an orbit of $\psi^\tau$,
  
  \item $\psi^\tau$ is transverse to $V:= \R \times \{0\} \times (0,\infty)$ and any orbit in $V$ returns to $V$ after a uniformly bounded, positive time,
  
  \item if $X=(X_t,X_s,X_r)(t,s,r)$ is the generator of $\psi^\tau$, then
  \[ \textstyle\frac{d}{dr}\big|_{r=0} X_s (t,0,r) \neq 0 \quad \forall t\in \R. \]
 \end{enumerate}
 Then the first-return time $\tau:V \to \R$ extends to a $C^\infty$ function $\tau: \overline V\to\R$.
\end{lemma}

The author is thankful to Umberto Hryniewicz for explaining to him the following proof of Lemma \ref{lemma umberto}.

\begin{proof}
 Set
 \[ F:\R\times \R^2 \to \R, \qquad F(\tau,t,r) := (\psi^\tau (t,0,r))_s , \]
 where $(.)_s:\R^3\to\R$ is the projection onto the $s$-coordinate. Then
 \[ \psi^\tau(t,0,r)\in V \qquad \iff \qquad F(\tau,t,r)=0 \quad \& \quad (\psi^\tau(t,0,r))_r > 0. \]
 By $\gam(\tau)$ being an orbit of $\psi^\tau$, we find $F(\tau,t,0)\equiv 0$ and hence we can write
 \[ F(\tau,t,r) = r \cdot G(\tau,t,r)  \]
 with
 \[ G(\tau,t,r) = \begin{cases} \frac{1}{r}F(\tau,t,r) & : r \neq 0 \\ \frac{\partial}{\partial r}\big|_{r=0} F(\tau,t,r) & : r = 0 \end{cases}. \]
 By Lemma \ref{lemma G smooth} in Appendix \ref{appendix}, $G$ is again a $C^\infty$ function. Observe that now
 \[ \psi^\tau(t,0,r)\in V \qquad \iff \qquad G(\tau,t,r)=0 \quad \& \quad ( \psi^\tau(t,0,r) )_r > 0. \]
 By assumption (2) we have a bounded $C^\infty$ function $\tau=\tau(t,r):V\to\R$ given by the first-return time under $\psi^\tau$ to $V$, which solves
 \[ \qquad G(\tau(t,r),t,r)=0 \quad \& \quad ( \psi^{\tau(t,r)}(t,0,r))_r > 0 \quad \forall(t,0,r)\in V. \]
 Now observe, that by definition of $F,G$
 \begin{align*}
  \frac{d}{d\tau}\bigg|_{\tau=0} G(\tau,t,0) & = \frac{d}{d\tau}\bigg|_{\tau=0} \frac{\partial}{\partial r}\bigg|_{r=0} ( \psi^\tau (t,0,r) )_s = \frac{\partial}{\partial r}\bigg|_{r=0} X_s(t,0,r) \neq 0
 \end{align*}
 by assumption (3). Hence, by the implicit function theorem and the boundedness of $\tau$, we can extend $\tau$ into a neighborhood of $V$ in $\R\times\{0\}\times\R$ and the lemma follows.
\end{proof}

We can thus prove Theorem \ref{thm sphere} from the introduction.

\begin{proof}[Proof of Theorem \ref{thm sphere}]
 By Lemmata \ref{lemma birkhoff} and \ref{birkhoff smooth}, the geodesic flow of $F$ can be reduced to the Birkhoff annulus map $\phi$, which is smooth in the closed annulus $A$. The topological entropy of $\phi$ vanishes by Remark \ref{entropy poincare map} (note that the first-return time is uniformly bounded). Hence, we are in the setting of Section \ref{section Fr-Ha} and obtain a homeomorphism $\psi\in\Diff(\s,P,\mu)$ with $\card P = 2$. The assumed existence of a second closed geodesic leads to an interior periodic point of the Birkhoff annulus map of period $q\geq 1$, say, and hence $\psi^q$ has a third fixed point in $\s-P$. We apply Theorems \ref{thm 1.2 Fr/Ha}, \ref{thm 1.4 Fr/Ha} to $\psi^q$ and let $U\in \A_q$ be a $\psi^q$-invariant, open annulus. Considering the rotation number $\rho: \AA\to \R/\Z$ of the annular compactification $h_U$ of $\psi^q$ as a non-constant, continuous function, we observe that in every interval of $\R/\Z$ there exists some $r$ with $\mu(\rho^{-1}(r))=0$, as the sets $\rho^{-1}(r)$ are disjoint and $\mu(U)<\infty$. We obtain $q+1$ (in fact, infinitely many) disjoint, open, $\psi^q$-invariant subsets $U_0,...,U_1\subset U$ as preimages of the rotation number of open intervals bounded by $r$'s with $\mu(\rho^{-1}(r))=0$, such that also the closures $\overline{U_i}$ are disjoint and $\mu(\overline{U_i})=\mu(U_i)$. Assuming the $U_i$ to be ordered according to area, $\mu(U_0)$ being the smallest, the set $\cup_{i=0}^{q-1} \psi^i(\overline{U_0})$ is closed, $\psi$-invariant, has non-empty interior and
 \[\mu(\cup_{i=0}^{q-1} \psi^i(\overline{U_0})) \leq q \cdot \mu(\overline{U_0}) < (q+1)\cdot \mu( U_0)\leq \mu(\cup_{i=0}^q U_i) \leq \mu(A). \]
 But if there exists a dense geodesic, then $\psi$ possesses a dense orbit, and we have $\cup_{i=0}^{q-1} \psi^i(\overline{U_0}) = A$, contradiction.
\end{proof}

\begin{remark}
 By the results of A. Harris and G. Paternain \cite{harris-pat}, also for $1/4$-pinched non-reversible Finsler metrics and more generally for dynamically convex Reeb flows on $\mathbb{S}^3$ there exist well-behaved, disc-like, global \Poincare surfaces of section. One can probably also prove smoothness of the arising first-return maps on the closure of the disc-like \Poincare surface and then apply the results of J. Franks and M. Handel. Hence, it is quite possible that Theorem \ref{thm sphere} generalizes to $1/4$-pinched non-reversible Finsler metrics and dynamically convex Reeb flows on $\mathbb{S}^3$.
\end{remark}

\section{Non-ergodicity in elliptic tubes for the 2-torus}\label{section torus}

We fix a reversible Finsler metric $F$ on $\T$ with geodesic flow $\phi_F^t:S\T\to S\T$ and assume $\h(\phi_F^t)=0$. In order to prove Theorem \ref{thm torus}, we construct a \Poincare section for the geodesic flow $\phi_F^t:S\T\to S\T$ of $F$, associated to a rational direction $\rho\in S^1$. For this, we use the reversibility of $F$.

Let $z=(z_1,z_2)\in\Z^2-\{0\}$ and $z^\perp := (-z_2,z_1)$. Choose a minimal axis $c^\perp:\R\to\R^2$ of the translation $\R^2\to\R^2$ associated to $z^\perp$ and consider the torus
\[ T^2_z := \R^2/(z\Z \oplus z^\perp \Z) . \]
Writing $\skp$ for the euclidean inner product on $\R^2 \cong T_x T_z^2$ and $\pi: T\T\to\T$ for the canonical bundle projection, we consider the open annulus
\[ A_z  := \{ v\in ST_z^2 ~|~ \exists t\in \R : \pi v = c^\perp(t) , ~ \la v, \dot c^\perp(t) \ra >0 \} \subset ST_z^2. \]
Note that $A_z \cong \R/\Z\times \R$. If $v\in S\T$, we take any lift $\widetilde{c_v}:\R\to\R^2$ of the geodesic $c_v:\R\to\T$ and set
\[ \rho(v) = \lim_{t\to\infty}\frac{\widetilde{c_v}(t)}{|\widetilde{c_v}(t)|} \quad \in S^1. \]
Due to $\h(\phi_F^t)=0$, it follows from Theorem \ref{thm torus htop=0}, that $\rho(v)$ exists for every $v\in S\T$, is independent of the choice of the lift $\widetilde{c_v}$ and $\rho(-v)=-\rho(v)$.

\begin{lemma}\label{C_z transverse}
 $A_z$ is transverse to the geodesic flow $\phi_F^t$. If $\h(\phi_F^t)=0$, then every orbit $\phi_F^tv$ with $v\in S\T$ and $\rho(v)$ lying in the connected component of $z$ in $S^1-\{\pm z^\perp\}$ hits $A_z$ after finite positive and negative time.
\end{lemma}

\begin{proof}
 Let $v\in A_z$, then the geodesic $c_v(t)=\pi(\phi^t_Fv)$ is transverse in $t=0$ to $\pi(A_z)=c^\perp(\R)$, showing transversality. The second claim follows directly from Theorem \ref{thm torus htop=0} (4).
\end{proof}

The first-return map to the \Poincare section $A_z$ is a $C^\infty$ diffeomorphism
\[ \phi: A_z\to A_z, \qquad A_z \cong \R/\Z\times \R, \]
preserving a smooth area $\nu$ (by being a first-return map of a flow conjugated to a Hamiltionian flow in $T^*T_z^2$) and we can see a transverse version of Figure \ref{fig htop=0} for $\phi$ in $A_z$. Recall the notation $\Gamma_\rho^\pm$ in Theorem \ref{thm torus htop=0} for the two invariant graphs in $S\T$ with asymptotic direction $\rho\in S^1$ and that for $\rho$ with irrational slope, we have $\Gamma_\rho^-=\Gamma_\rho^+=:\Gamma_\rho$. We use the same notation for the intersections of these graphs with $A_z$. The $\phi_F^t$-invariant tori given by Theorem \ref{thm torus htop=0} then appear as $\phi$-invariant, (Lipschitz) continuous, simple, closed and non-contractible curves $\Gamma_\rho^\pm\subset A_z$.

\begin{lemma}\label{irrational boundaries}
 There exist $\rho_-,\rho_+\in S^1$ with irrational slope, such that $\rho_-<z/|z|<\rho_+$ in the counterclockwise orientation of $S^1$ and such that $\phi:A_z\to A_z$ has no fixed points in $A_z$ between the invariant graphs $\Gam_{\rho_-}, \Gam_{\rho_+}$ other than the ones in the region enclosed by the two graphs $\Gam_{z/|z|}^\pm$.
\end{lemma}

\begin{proof}
 If we choose a neighborhood of $z$ in $z+\R z^\perp$ that intersects $\Z^2$ only in $z$, we can choose irrational $\rho_\pm$ such that the lines $\R_{>0}\rho_\pm$ intersect $z+\R z^\perp$ in that neighborhood on either side of $z$. Let $v\in A_z$ be a fixed point for $\phi$, then $c_v:\R\to \T$ is a closed geodesic with some homotopy class $z+\lam z^\perp\in \Z^2$ for some $\lam\in \Z$ and if $\lam\neq 0$, the asymptotic direction $\rho(v)= \frac{z+\lam z^\perp}{|z+\lam z^\perp|}$ lies outside the segment between $\rho_\pm$ by construction. On the other hand, if the orbit $\dot c_v$ lies between $\Gam_{\rho_\pm}$, then $\rho(v)$ lies between $\rho_\pm$ in $S^1$, so $\lam=0$ and hence $\rho(v)=z/|z|$ and $\dot c_v$ lies between $\Gam_{z/|z|}^\pm$.
\end{proof}

We want to study the behavior of $\phi:A_z\to A_z$ in the space between the $\phi$-invariant Lipschitz curves
\[ \Gam_\pm := \Gam_{z/|z|}^\pm \subset A_z. \]
By Lemma \ref{irrational boundaries}, we are also given two disjoint, $\phi$-invariant Lipschitz curves
\[ \gam_\pm := \Gam_{\rho_\pm} \subset A_z \]
and we restrict ourselves to the subset $N$ of $A_z$ between $\gam_-$ and $\gam_+$, restricting the area $\nu$ to $\Int N$. By Remark \ref{entropy poincare map}, we find
\[ \h(\phi|_N) \leq \h(\phi_F^t)=0 . \]
Collapsing $\gam_\pm$ into two points $p_\pm\in\s$, we are precisely in the situation of Lemmata \ref{lemma mu}, \ref{FrHa works on T^2} and hence can apply Theorems \ref{thm 1.2 Fr/Ha}, \ref{thm 1.4 Fr/Ha} of Franks and Handel in Section \ref{section Fr-Ha}. Analogous to Definition \ref{def elliptic tubes}, we call the components $E$ of $N-\cup\Gam_\pm$ between $\Gam_\pm$ elliptic islands. Since the non-empty intersection $\cap\Gam_\pm$ consists of fixed points for $\phi$, elliptic islands are $\phi$-invariant (instead of merely being permuted). Cf. Figure \ref{fig elliptic_island} for the notation.

\begin{figure}[!htb]\centering
\includegraphics[scale=0.7]{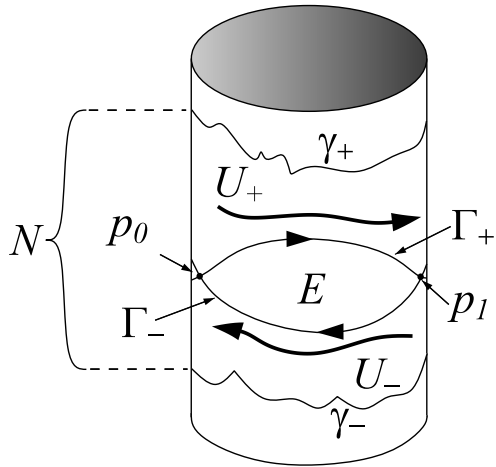}
\caption{The annulus $A_z$ with the curves $\gam_\pm,\Gam_\pm$ and fixed points $p_0,p_1$ in the boundary of an elliptic island $E$. The arrows indicate the principle direction of $\phi$. \label{fig elliptic_island}}
\end{figure}

Next, we observe that the regions above $\Gam_+$ and below $\Gam_-$ in $N$ actually occur as $\phi$-invariant, open annuli in Theorem \ref{thm 1.2 Fr/Ha}. As in Section \ref{section Fr-Ha}, we write
\[ M=N-\Fix(\phi), \qquad f=\phi|_M . \]
We will use the following fact, cf. Theorem (2.1) in \cite{franks}. Namely, if $h:\Int \AA\to \Int \AA$ is a fixed-point free, orientation- and area-preserving homeomorphism of the open annulus, then the set of points with vanishing rotation number has measure zero.

\begin{lemma}\label{annulus in island}
The open annuli $U_-$ between $\gam_-$ and $\Gam_-$ and $U_+$ between $\Gam_+$ and $\gam_+$ belong the the collection $\A$ of maximal $f$-invariant, open annuli in Theorem \ref{thm 1.2 Fr/Ha}.
\end{lemma}

\begin{proof}
By Lemma \ref{irrational boundaries}, both $U_\pm$ belong to $M$ and by invariance of $\gam_\pm,\Gam_\pm$, they are $f$-invariant. By Theorem \ref{thm 1.2 Fr/Ha}, there exists an $f$-invariant, open annulus $U\in\A$, such that $U\supset U_+$. Suppose that $U\cap\Gam_+\neq \emptyset$, and note that every point $x$ below $\Gam_+$ has zero rotation number in the annulus $U$, since it is contained in an elliptic island. But then $f|_U$ has a fixed point by the fact recalled above, contradicting $U\subset M$.
\end{proof}

We can now prove Theorem \ref{thm torus}. In the following proof, we fix an elliptic island $E\subset N$, which has two fixed points $p_0,p_1\in\partial E$, cf. Figure \ref{fig elliptic_island}.

\begin{proof}[Proof of Theorem \ref{thm torus}]
For item (1), note that we can restrict ourselves to the set $E\cap M$. Since the set of weakly free disc recurrent points $\W\cap E$ has full measure in $E\cap M$ and since the set of maximal invariant annuli $\A$ in Theorem \ref{thm 1.2 Fr/Ha} is countable, we can even restrict to some $U\in\A$ lying in $E$ by Lemma \ref{annulus in island}. If $x\in U$ with $\om(\phi,x)\cap\partial E\neq \emptyset$, then by the heteroclinic dynamics in $\partial E-\{p_0,p_1\}$, we obtain $\om(\phi,x)\cap \{p_0,p_1\}\neq \emptyset$. Hence, one point $p\in\{p_0,p_1\}$ lies in one end $\partial_1 U$ in $N$. On the other hand, $\partial_1 U$ cannot consist only of $p$, since $U\subset \Int E$ by Lemma \ref{annulus in island}. Thus, Lemma \ref{lemma prime ends} shows that in the annular compactification $h_U:\AA\to\AA$ we find a fixed point $\tilde p$ of $h_U$ in $\partial_1\AA$ corresponding to $\partial_1U$ and hence the rotation number of $h_U|_{\partial_1\AA}$ vanishes. If the above $x\in U$ corresponds to some $\tilde x\in \AA$, then we find $\partial_1\AA \cap \om(h_U,\tilde x) \neq \emptyset$ and by continuity of the rotation number in Theorem \ref{thm 1.4 Fr/Ha}, we have $\rho_{h_U}(\tilde x)=0$. But by the above stated special case of Theorem (2.1) in \cite{franks}, this can happen only for a set of points $\tilde x\in\AA$ (and hence $x\in U$) of measure zero, for the homeomorphism $h_U|_{\Int\AA}$ is fixed-point free by definition. This proves (1).

For (2) observe first that a dense orbit for the geodesic flow $\phi_F^t|_E$ will be dense for the \Poincare map $\phi|_E$ and hence also be dense in some $U\in\A$ by Lemma \ref{annulus in island}. Then Theorem \ref{thm 1.4 Fr/Ha} (1) shows that $\rho_{h_U}$ is constant, while $\psi$ has at least three fixed points $p\in\{p_0,p_1\}$ and $p_-,p_+$ corresponding to $\gamma_-,\gamma_+$. Hence $\rho_{h_U}$ is non-constant by Theorem \ref{thm 1.4 Fr/Ha} (2), contradiction.
\end{proof}

\newpage

\appendix

\section{}\label{appendix}

In Section \ref{section sphere} we used the following lemma, which we will prove here.

\begin{lemma}\label{lemma G smooth}
 Let $F:\R^n \to\R$ be a $C^\infty$ function, such that writing $(x,t)\in \R^{n-1}\times \R$ we have $F(x,0)\equiv 0$. Set
 \[ G:\R^n\to\R, \quad G(x,t):= \begin{cases} F(x,t)/t &: t\neq 0 \\ \partial_t F(x,0) &: t =0 \end{cases} . \]
 Then $G:\R^n\to\R$ is $C^\infty$ as well.
\end{lemma}

\begin{proof}
 A Taylor expansion of $\partial_t^kF$ with $k\in \N_0$ in the $t$-variable shows
 \begin{align}\label{eqn F taylor}
  & ~ \partial_t^kF(x,t) \\ \nonumber
  = & ~ \partial_t^kF(x,0) + t \partial_t^{k+1}F(x,0) + \frac{t^2}{2} \partial_t^{k+2}F(x,0) + \frac{t^3}{6} \partial_t^{k+3}F(x,\tau)
 \end{align}
 for some $\tau \in [-1,1]$, if $|t|\leq 1$. Via induction on $k$ one easily shows
 \begin{align} \label{eqn G ableitung}
  \partial_t^k G(x,t) = \frac{\partial_t^k F(x,t) -k \partial_t^{k-1}G(x,t)}{t} \qquad \text{in } \{t\neq 0\}.
 \end{align}
 We will prove via induction on $k$ that $G$ is a $C^k$-function with $\partial_t^kG(x,0)=\frac{1}{k+1}\partial^{k+1}_tF(x,0)$. For $k=0$ this is clear from the definition of $G$. Assume now that the statement is true for $k-1,k$. We then find using \eqref{eqn F taylor}, \eqref{eqn G ableitung}
 \begin{align*}
 &~ \frac{\partial_t^kG(x,t)-\partial_t^kG(x,0)}{t} \\
 =&~ \frac{\partial_t^k F(x,t) -k \partial_t^{k-1}G(x,t)-t  \partial_t^kG(x,0)}{t^2} \\
 =&~ \frac{ \left\{ \begin{array}{c} \partial_t^kF(x,0) + t \partial_t^{k+1}F(x,0) + \frac{t^2}{2} \partial_t^{k+2}F(x,0) + \frac{t^3}{6} \partial_t^{k+3}F(x,\tau) \\ -k \partial_t^{k-1}G(x,t)-t  \partial_t^kG(x,0) \end{array} \right\} }{t^2} \\
 =&~ \frac{ \left\{ \begin{array}{c} \partial_t^kF(x,0) -k \partial_t^{k-1}G(x,t) \\ + t\cdot  \big( \partial_t^{k+1}F(x,0) - \partial_t^kG(x,0) \big) \end{array} \right\} }{t^2}  + \frac{1}{2} \partial_t^{k+2}F(x,0) + \frac{t}{6} \partial_t^{k+3}F(x,\tau) .
 \end{align*}
 By our induction hypothesis for $k-1$, we have $\partial_t^kF(x,0) -k \partial_t^{k-1}G(x,t)\to 0$ as $t\to 0$ and applying the l'Hospital rule and $\partial_t^kG(x,0)=\frac{1}{k+1}\partial^{k+1}_tF(x,0)$, the above fraction tends to
 \begin{align*}
   &~ \lim_{t\to 0} \frac{ \partial_t^{k+1}F(x,0) - \partial_t^kG(x,0) -k \partial_t^{k}G(x,t) }{2t} \\
   = &~ \lim_{t\to 0} \frac{ \partial_t^{k+1}F(x,0) - (k+1)\partial_t^kG(x,0) +k \partial_t^{k}G(x,0)-k \partial_t^{k}G(x,t) }{2t} \\
   = &~ \frac{-k}{2} \cdot \lim_{t\to 0} \frac{ \partial_t^{k}G(x,t) - \partial_t^{k}G(x,0) }{t} .
 \end{align*}
 Hence, we obtain the existence of the limit
 \begin{align*}
  (2+k) \cdot \partial_t^{k+1}G(x,0) = (2+k) \cdot \lim_{t\to 0} \frac{\partial_t^kG(x,t)-\partial_t^kG(x,0)}{t} = \partial_t^{k+2}F(x,0) .
 \end{align*}
 Using \eqref{eqn G ableitung}, the induction hypothesis for $k$, the l'Hospital rule and our just found formula for $\partial_t^{k+1}G(x,0)$, we obtain
 \begin{align*}\label{G ableitung stetig}
  \lim_{t\to 0} \partial_t^{k+1} G(x,t) &= \lim_{t\to 0} \frac{\partial_t^{k+1} F(x,t) -(k+1) \partial_t^k G(x,t)}{t} \\ \nonumber
  &= \partial_t^{k+2} F(x,0) - (k+1) \partial_t^{k+1} G(x,0) = \partial_t^{k+1}G(x,0).
 \end{align*}
 The continuity of $\partial_t^{k+1} G$ in $x$ at $t=0$ is obvious from $F$ being $C^{k+2}$. To check the other partial derivatives of $G$ not involving $t$, observe
 \[ \partial_{x_{i_1}}\cdots \partial_{x_{i_l}} G(x,t) = \begin{cases} \frac{1}{t} \partial_{x_{i_1}}\cdots \partial_{x_{i_l}} F(x,t) & : t\neq 0 \\  \partial_t \partial_{x_{i_1}}\cdots \partial_{x_{i_l}} F(x,0) & : t = 0 \end{cases} . \]
 This function is continuous, since $\partial_{x_{i_1}}\cdots \partial_{x_{i_l}} F(x,0) \equiv 0$ due to $F(x,0)\equiv 0$. To check derivatives of the form $\partial_t^{k+1-l} \partial_{x_{i_1}}\cdots \partial_{x_{i_l}} G$, apply the above formula for $\partial_{x_{i_1}}\cdots \partial_{x_{i_l}} G$ and observe that we can apply the above arguments to $F$ replaced by the functions $\partial_{x_{i_1}}\cdots \partial_{x_{i_l}} F$. Having this in mind, observe that by \eqref{eqn G ableitung} and our formula for $\partial_tG(x,0)$
 \begin{align*}
  \partial_{x_j} \partial_t G(x,t) & = \begin{cases} \partial_{x_j} \frac{\partial_t F(x,t) - \frac{1}{t}F(x,t)}{t} & : t\neq 0 \\ \partial_{x_j} \frac{1}{2} \partial_t^2 F(x,0) & : t = 0 \end{cases} \\
  & = \begin{cases} \frac{\partial_t (\partial_{x_j} F)(x,t) - \frac{1}{t}(\partial_{x_j} F)(x,t)}{t} & : t\neq 0 \\ \frac{1}{2} \partial_t^2 (\partial_{x_j} F)(x,0) & : t = 0 \end{cases} \\
  & = \partial_t (\textstyle \frac{1}{t}\partial_{x_j} F)(x,t) \\
  & = \partial_t \partial_{x_j} G (x,t) .
 \end{align*}
 This shows that we can exchange $\partial_t$ and $\partial_{x_j}$ applied to $G$, so that we can always obtain the form $\partial_t^{k+1-l} \partial_{x_{i_1}}\cdots \partial_{x_{i_l}} G$ of the partial derivatives. Thus, $G$ is also $C^{k+1}$.
\end{proof}

\abs
{\footnotesize
The author's preprint \cite{paper2} is available online at

\url{http://www.ruhr-uni-bochum.de/ffm/Lehrstuehle/Lehrstuhl-X/jan.html}.
}

\end{document}